\documentclass[12pt]{article}
\usepackage{tikz}
\usepackage{hyperref}
\usepackage{empheq}
\usepackage[shortlabels]{enumitem}
\setlist[enumerate]{nosep}
\usepackage[doc]{optional}
\usepackage{xcolor}
\usepackage{float}
\usepackage{soul}
\usepackage{graphicx}
\definecolor{labelkey}{rgb}{0,0.08,0.45}
\definecolor{refkey}{rgb}{0,0.6,0.0}
\definecolor{Brown}{rgb}{0.45,0.0,0.05}
\definecolor{lime}{rgb}{0.00,0.8,0.0}
\definecolor{lblue}{rgb}{0.5,0.5,0.99}

 \usepackage{mathpazo}

\usepackage{stmaryrd}
\usepackage{amssymb}
\newcommand{\seppfive}{\setlength{\itemsep}{-5pt}}

\newcommand{\SE}{\ensuremath{{\mathcal S}}}

\providecommand{\siff}{\Leftrightarrow}

\newcommand{\nnn}{\ensuremath{{n\in{\mathbb N}}}}

\newcommand{\menge}[2]{\big\{{#1}~\big |~{#2}\big\}}

\newcommand{\To}{\ensuremath{\rightrightarrows}}

\newcommand{\fenv}[1]%
{\ensuremath{\,\overrightarrow{\operatorname{env}}_{#1}}}
\newcommand{\benv}[1]%
{\ensuremath{\,\overleftarrow{\operatorname{env}}_{#1}}}

\newcommand{\RR}{\ensuremath{\mathbb R}}

\newcommand{\NN}{\ensuremath{\mathbb N}}

\newcommand{\dom}{\ensuremath{\operatorname{dom}}}

\newcommand{\ran}{\ensuremath{\operatorname{ran}}}

\newcommand{\Fix}{\ensuremath{\operatorname{Fix}}}
\newcommand{\Id}{\ensuremath{\operatorname{Id}}}
\newcommand{\JAinv}{J_{A^{-1}}}

\newcommand{\TAB}{T_{(A,B)}}

\providecommand{\gap}{v}

\providecommand{\fejer}{Fej\'{e}r~}
{\begin{list}{}{%
\settowidth{\labelwidth}{\textrm{#1~}}%
\setlength{\leftmargin}{\labelwidth+\labelsep}}}%requires macro calc.sty
{\end{list}}
\usepackage{amsthm}
\usepackage[capitalize]{cleveref}
\crefname{equation}{}{equations}
\crefname{chapter}{Appendix}{chapters}
\crefname{item}{}{items}
\newtheorem{theorem}{Theorem}[section]
\newtheorem{lemma}[theorem]{Lemma}
\newtheorem{lem}[theorem]{Lemma}

\newtheorem{cor}[theorem]{Corollary}

\newtheorem{prop}[theorem]{Proposition}
\newtheorem{definition}[theorem]{Definition}

\newtheorem{thm}[theorem]{Theorem}%[section]
\newtheorem{example}[theorem]{Example}
\newtheorem{ex}[theorem]{Example}
\newtheorem{fact}[theorem]{Fact}

\newtheorem{rem}[theorem]{Remark}
\providecommand{\ds}{\displaystyle}

\providecommand{\norm}[1]{\lVert#1\rVert}
\providecommand{\normsq}[1]{\lVert#1\rVert^2}
\providecommand{\bk}[1]{\left(#1\right)}
\providecommand{\stb}[1]{\left\{#1\right\}}
\providecommand{\innp}[1]{\langle#1\rangle}

\providecommand{\RA}{\Rightarrow}

\providecommand{\weak}{\rightharpoonup}

\providecommand{\RR}{\mathbb{R}}

\providecommand{\ran}{\operatorname{ran}}

\providecommand{\intr}{\operatorname{int}}

\providecommand{\dom}{\operatorname{dom}}
\providecommand{\parl}{\operatorname{par}}

\newcommand{\fix}{\ensuremath{\operatorname{Fix}}}

\providecommand{\gra}{\operatorname{gra}}
\providecommand{\Id}{\operatorname{{ Id}}}

\providecommand{\fady}{\varnothing}

\providecommand{\rras}{\rightrightarrows}
\providecommand{\To}{\rightrightarrows}

\providecommand{\NN}{\mathbb{N}}

\providecommand{\fix}{\operatorname{Fix}}
\providecommand{\ran}{\operatorname{ran}}

\providecommand{\Id}{\operatorname{Id}}

\providecommand{\pt}{{\partial}}

\providecommand{\inns}[2][w]{#2_{#1}}
\newcommand{\outs}[2][w]{{_#1}#2}

\providecommand{\fady}{\varnothing}

\providecommand{\RR}{\mathbb{R}}
\providecommand{\NN}{\mathbb{N}}

\providecommand{\xt}{x}
\providecommand{\yt}{y}
\providecommand{\ut}{a^*}
\providecommand{\vt}{b^*}
\providecommand{\at}{a}
\providecommand{\bt}{b}

\definecolor{myblue}{rgb}{.8, .8, 1}
  \newcommand*\mybluebox[1]{%
    \colorbox{myblue}{\hspace{1em}#1\hspace{1em}}}

\allowdisplaybreaks % or locally if problems {\allowdisplaybreaks
\begin{document}
%-------------------------------------------------------------------------

%\tikzstyle{decision} = [diamond, draw, fill=blue!50]
%\tikzstyle{line} = [draw, -stealth, thick]
%\tikzstyle{elli}=[draw, ellipse, fill=red!50,minimum height=8mm, text width=5em, text centered]
%\tikzstyle{block} = [draw, rectangle, fill=blue!50, text width=8em, text centered, minimum height=15mm, node distance=10em]
%
\title{\textsc
On the Douglas--Rachford algorithm}

%\author{
%Heinz H.\ Bauschke\thanks{
%Mathematics, University
%of British Columbia,
%Kelowna, B.C.\ V1V~1V7, Canada. E-mail:
%\texttt{heinz.bauschke@ubc.ca}.},
%Warren L.\ Hare\thanks{
%Mathematics,
%University of British Columbia,
%Kelowna, B.C.\ V1V~1V7, Canada.
%E-mail:  \texttt{warren.hare@ubc.ca}.},
%~and Walaa M.\ Moursi\thanks{
%Mathematics, University of
%British Columbia,
%Kelowna, B.C.\ V1V~1V7, Canada. E-mail:
%\texttt{walaa.moursi@ubc.ca}.}}
%

\title{\textsc
On the Douglas--Rachford algorithm}

\author{
Heinz H.\ Bauschke\thanks{
Mathematics, University
of British Columbia,
Kelowna, B.C.\ V1V~1V7, Canada. E-mail:
\texttt{heinz.bauschke@ubc.ca}.}
~~and Walaa M.\ Moursi\thanks{
Mathematics, University of British Columbia, Kelowna, B.C.\ V1V~1V7, Canada,
and 
Mansoura University, Faculty of Science, Mathematics Department, 
Mansoura 35516, Egypt. 
E-mail: \texttt{walaa.moursi@ubc.ca}.}}

\date{April 15, 2016}

\maketitle

\begin{abstract}
\noindent
The Douglas--Rachford algorithm 
is a very popular splitting technique 
for finding a zero of the sum of
two maximally monotone operators.
However, the behaviour of the algorithm 
remains mysterious in the general inconsistent  case, i.e.,
when the sum problem has no zeros.
More than a decade ago, however, it was shown 
that in the (possibly inconsistent) convex
feasibility setting, the shadow sequence
remains bounded and it is weak cluster 
points solve a best approximation problem.

In this paper, we advance the understanding of the inconsistent
case significantly by providing a complete proof
of the full weak convergence in the 
convex feasibility setting. In fact,
a more general sufficient condition for the 
weak convergence in the general case is presented.
Several examples illustrate the results. 
\end{abstract}
{\small

}

\section{Introduction}

In this paper we assume that
\begin{empheq}[box=\mybluebox]{equation*}
\label{T:assmp}
X \text{~~is a real Hilbert space},
\end{empheq} 
with inner product $\innp{\cdot,\cdot}$ and
induced norm $\norm{\cdot}$. 
A classical problem in optimization is to find a minimizer 
of the sum 
of two proper convex lower semicontinuous functions.
This problem can be modelled as 
\begin{equation}
\label{e:sumprob}
\text{find $x\in X$ such that $0\in (A+B)x$,}
\end{equation}
where $A$ and $B$ are maximally monotone operators on $X$,
namely the subdifferential operators of the functions under consideration. 
For detailed discussions on problem \cref{e:sumprob}
and the connection to optimization problems we
refer the reader to 
\cite{BC2011}, 
\cite{Borwein50}, 
\cite{Brezis}, 	
\cite{BurIus},
\cite{Comb96},
\cite{Simons1},
\cite{Simons2},
\cite{Rock98},
\cite{Zeidler2a},
\cite{Zeidler2b}, and the references therein.

Due to its general convergence results, the 
Douglas--Rachford algorithm has become a very popular 
splitting technique to solve 
the sum problem \cref{e:sumprob}
provided that the solution set is nonempty.
The algorithm was first introduced 
in \cite{DoRa} to numerically solve certain 
types of heat equations. 
Let $x\in X$, let $T=\TAB$ be the 
Douglas--Rachford operator associated
with the ordered pair $(A,B)$ (see \cref{def:T}) and let
$J_A$ be the resolvent of $A$ (see \cref{f:JA:RA}). 
In their masterpiece \cite{L-M79},
Lions and Mercier extended the algorithm 
to be able to find a zero 
of the sum of two,
not necessarily linear and possibly multivalued, maximally
monotone operators. 
They proved that the ``governing sequence" $(T^n x)_\nnn$
converges weakly to a fixed point of $T$,
and that if $A+B$ is maximally monotone,
then the weak cluster points of the ``shadow sequence''
$(J_AT^n x)_\nnn$ are solutions 
of \cref{e:sumprob}.
In \cite{Svaiter}, Svaiter provided a proof of the 
weak convergence of the shadow sequence, 
regardless of $A+B$. 

Nonetheless, very little is known about the 
behaviour of the algorithm in the inconsistent setting, 
i.e., when the set of zeros of the sum is empty. 
In \cite{BCL04} (see \cref{rem:BCL04}),
the authors showed that for the case when
$A$ and $B$ are normal cone operators of 
two nonempty closed convex subsets of $X$, and 
$P_{\overline{\ran}(\Id -T)}\in \ran (\Id-T)$ (see \cref{F:v:WD}),
then the shadow sequence $(J_AT^n x)_\nnn$
is bounded and its weak cluster points solve a certain
best approximation problem. 
 
In this paper 
we derive some new and useful identities 
for the Douglas--Rachford operator.
The main contribution of the paper 
is
generalizing the results in \cite{BCL04}
by proving the full weak convergence of the shadow sequence
in the convex feasibility setting (see \cref{thm:nc:shad}). 
While the general case case remains open (see
\cref{ex:gen:incon} and \cref{rem:gen:case}),
we provide some sufficient conditions for the convergence of
the shadow sequence in some special cases 
(see \cref{thm:abs:gen:shad}).

As a by product of our analysis we present a new proof
for the result in \cite{Svaiter} concerning 
the weak convergence of the shadow sequence 
(see \cref{thm:simp:pr:d}). Our proof
is in the spirit of the techniques used in \cite{L-M79}. 

The notation used in the paper is standard and follows largely
\cite{BC2011}. 

\section{Useful identities for the 
Douglas--Rachford operator}

We start with two elementary identities which are easily verified
directly. 

\begin{lem}
\label{lem:simple:mi}
Let $(a,b,z)\in X^3$. 
Then the following hold:
\begin{enumerate}
\item
\label{lem:simple:mi:i}
$\innp{z,z-a+b}=\normsq{z-a+b}+\innp{a,z-a}+\innp{b,2a-z-b}$.
\item
\label{lem:simple:mi:ii}
$\innp{z,a-b}=\normsq{a-b}+\innp{a,z-a}+\innp{b,2a-z-b}$.
\item
\label{lem:simple:mi:iii}
$\norm{z}^2=\normsq{z-a+b}+\normsq{b-a}
+2\innp{a,z-a}+2\innp{b,2a-z-b}.$
\end{enumerate}
\end{lem}

\begin{lem}
\label{lem:abs:8}
Let $(a,b,x,y,a^*,b^*,u,v)\in X^8$. 
Then  
\begin{align}
%\label{eq:maxm:ve}
 \innp{(\at,\bt)-(\xt, \yt),(\ut,b^*)-(u,v)}&
=\innp{\at-\bt,\ut}%-\innp{\bt,a^*_n}
+\innp{\xt,u}
-\innp{\xt,\ut}
-\innp{\at-\bt,u}%+\innp{\bt,\xt^*}
\nonumber\\
&\qquad +\innp{\bt,\ut+\vt}%+\innp{\bt,a^*_n}
+\innp{\yt,v}
-\innp{\yt,\vt}
-\innp{\bt,u+v}.%-\innp{b_n,x_a^*}
%\nonumber\\
\end{align}
\end{lem}

Unless stated otherwise, we assume from now on that 
\begin{empheq}[box=\mybluebox]{equation*}
A:X\rras X \text{~~and~~} B:X\rras X \text{~~are maximally 
monotone operators.}
\end{empheq} 
The following result
concerning the 
\emph{resolvent} 
$J_A:=(\Id+A)^{-1}$ and 
the \emph{reflected resolvent} 
$R_A:=2J_A-\Id$ is well known;
see, e.g., \cite[Corollary~23.10(i)\&(ii)]{BC2011}.

\begin{fact}
\label{f:JA:RA}
$J_A:X\to X$ is firmly nonexpansive
and 
$R_A:X\to X$ is nonexpansive.
\end{fact}
Let us recall the well-known 
inverse resolvent identity (see, \cite[Lemma~12.14]{Rock98}):
\begin{equation}%\label{inv:res}
\label{e:iri}
J_A+J_{A^{-1}}=\Id,
\end{equation}
and the following useful description of the graph of $A$. 

 \begin{fact}[\bf Minty parametrization]{\rm(See \cite{Minty}.)}
\label{f:Mintypar}
$M\colon X\to \gra A:x\mapsto(J_A x,J_{A^{-1}}x)$
is a continuous bijection, with continuous inverse
$M^{-1}\colon \gra A \to X :(x,u)\to x+u$ ; consequently, 
\begin{equation}\label{Min:par}
\gra A=M(X) = \menge{(J_A x,x-J_A x)}{x\in X}.
\end{equation}
\end{fact}
\begin{definition}
%\label{def:T}
The \emph{Douglas--Rachford splitting operator} associated with 
$(A,B)$ is
\begin{empheq}[box=\mybluebox]{equation}
\label{def:T}
T:=\TAB=\tfrac{1}{2}(\Id+R_BR_A)=\Id-J_A+J_B R_A.
\end{empheq}
\end{definition}

We will simply use $T$ instead of $\TAB$ provided there is no
cause for confusion.

The following result will be useful. 

\begin{lem}
\label{lem:cluster:fen}
Let $x\in X$. Then the following hold:
\begin{enumerate}
\item
\label{lem:cluster:fen:i}
$x-Tx=J_A x -J_BR_A x
=J_{A^{-1}} x +J_{B^{-1}}R_A x$.
\item
\label{lem:cluster:fen:ii}
$(J_A x,J_BR_Ax,J_{A^{-1}} x,
J_{B^{-1}}R_A x)$
lies in $\gra (A\times  B)$.
\end{enumerate}
\end{lem}
\begin{proof}
\ref{lem:cluster:fen:i}:
The first identity is 
a direct consequence of
\cref{def:T}. 
In view of \cref{e:iri}
$J_A x -J_BR_A x
-J_{A^{-1}} x -J_{B^{-1}}R_A x=
J_A x -(x-J_{A} x)- (J_B+J_{B^{-1}})R_A x
=R_A x -R_Ax=0
$, which proves the second identity.
\ref{lem:cluster:fen:ii}:
Use %Minty's parametrization 
\cref{e:prod:resolvent} and \cref{f:Mintypar}  
applied to
 $A\times B$ at $(x,R_Ax)\in X\times X$.
\end{proof}

The next theorem is a direct consequence
of the key identities presented in \cref{lem:simple:mi}.
\begin{thm}
\label{thm:simp:pr}
Let $x\in X$ and let $y\in X$.
Then  the following hold:
\begin{enumerate}
\item
$
\begin{aligned}[t]
\innp{Tx-Ty,x-y}&=\normsq{Tx-Ty}+\innp{J_Ax-J_Ay,\JAinv x-\JAinv y}\\
&\qquad +\innp{J_BR_Ax-J_BR_Ay,J_{B^{-1}}R_A x-J_{B^{-1}}R_A y}.
\end{aligned}
$
\label{thm:simp:pr:-i}
\item
\label{thm:simp:pr:-ii}
$
\begin{aligned}[t]
\innp{(\Id-T)x-(\Id-T)y,x-y}&=\normsq{(\Id-T)x-(\Id-T)y}
+\innp{J_Ax-J_Ay,\JAinv x-\JAinv y}\\
&\qquad+\innp{J_BR_Ax-J_BR_Ay,J_{B^{-1}}R_A x-J_{B^{-1}}R_A y}.
\end{aligned}
$

\item
\label{thm:simp:pr:i}
$
\begin{aligned}[t]
\normsq{x-y}&=\normsq{Tx-Ty}+\normsq{(\Id-T)x-(\Id-T)y}
+2\innp{J_Ax-J_Ay,\JAinv x-\JAinv y}\\
&\qquad+2\innp{J_BR_Ax-J_BR_Ay,J_{B^{-1}}R_A x-J_{B^{-1}}R_A y}.
\end{aligned}
$
\item
\label{thm:simp:pr:ii}
$
\begin{aligned}[t]
&\norm{J_{A} x-J_{A}y}^2+\norm{J_{A^{-1}} x-J_{A^{-1}}y}^2
- \norm{J_{A}T x-J_{A}Ty}^2-\norm{J_{A^{-1}} Tx-J_{A^{-1}} Ty}^2\\
&=%\normsq{Tx-Ty}+
\normsq{(\Id-T)x-(\Id-T)y}
+2\innp{J_A Tx-J_A Ty,\JAinv Tx-\JAinv Ty}\\
&\qquad+2\innp{J_BR_Ax-J_BR_Ay,J_{B^{-1}}R_A x-J_{B^{-1}}R_A y}.
\end{aligned}
$ 

\item
\label{thm:simp:pr:vv}
%\begin{equation}\label{e:ine:simp}
 $
 \norm{J_{A}T x-J_{A}Ty}^2+\norm{J_{A^{-1}} Tx-J_{A^{-1}} Ty}^2
 \le \norm{J_{A} x-J_{A}y}^2+\norm{J_{A^{-1}} x-J_{A^{-1}}y}^2.
$

\end{enumerate}
\end{thm}
\begin{proof}
\ref{thm:simp:pr:-i}--\ref{thm:simp:pr:i}:
Use Lemma~\ref{lem:simple:mi}--\ref{lem:simple:mi:iii}
respectively,
with $z=x-y$, $a=J_A x-J_Ay$ and
$b=J_BR_A x-J_BR_A y$, \cref{Min:par} 
and \cref{def:T}.
\ref{thm:simp:pr:ii}:
It follows from the 
\cref{e:iri} that %$(\forall x\in X)$ $(\forall y\in X)$ we have
\begin{subequations}
\label{JA:decom}
\begin{align}
\norm{x-y}^2&=\norm{J_Ax-J_Ay+J_{A^{-1}}{x}-J_{A^{-1}}{y}}^2\\
&=\norm{J_Ax-J_Ay}^2+\norm{J_{A^{-1}}{x}-J_{A^{-1}}{y}}^2
+2\innp{J_Ax-J_Ay,J_{A^{-1}}{x}-J_{A^{-1}}{y}}.
\end{align}
\end{subequations}
Applying \eqref{JA:decom} 
%to $x$ and $y$
%and 
to $(Tx,Ty)$ instead of $(x,y)$  yields
\begin{align}\label{e:xy:TxTy}
\norm{Tx-Ty}^2
&=\norm{J_ATx-J_ATy}^2+\norm{J_{A^{-1}}{Tx}-J_{A^{-1}}{Ty}}^2\notag\\
&\qquad+2\innp{J_ATx-J_ATy,J_{A^{-1}}{Tx}-J_{A^{-1}}{Ty}}.
\end{align}
Now combine \cref{JA:decom}, \eqref{e:xy:TxTy} 
and \ref{thm:simp:pr:i}
to obtain \ref{thm:simp:pr:ii}. 
\ref{thm:simp:pr:vv}: 
In view of 
\eqref{Min:par}, the monotonicity of $A$
and $B$ implies
$\innp{J_ATx-J_ATy,\JAinv Tx-\JAinv Ty}\ge 0$ and 
$\innp{J_BR_Ax-J_BR_Ay,J_{B^{-1}}R_A x-J_{B^{-1}}R_A y}\ge 0$.
Now use \ref{thm:simp:pr:ii}.
 \end{proof}

\section{The Douglas--Rachford operator, Attouch--Th\'era duality and
solution sets}

The \emph{Attouch--Th\'{e}ra dual pair} (see \cite{AT}) of $(A,B)$
is $(A,B)^*:=(A^{-1},B^{-\ovee})$,
where 
\begin{equation}
B^{\ovee}:= (-\Id)\circ B\circ(-\Id)\quad{\text{and}}
\quad B^{-\ovee}:=(B^{-1})^\ovee=(B^\ovee)^{-1}.
\end{equation}
We use
\begin{equation}
Z:=Z_{(A,B)}=(A+B)^{-1}(0)
\qquad\text{and }\qquad 
K:=K_{(A,B)}=(A^{-1}+B^{-\ovee})^{-1}(0)
\end{equation}
to denote the primal and dual solutions, respectively
(see, e.g., \cite{JAT2012}).

Let us record some useful properties of $T_{(A,B)}$.

\begin{fact}
\label{f:sd:ZK}
The following hold:
\begin{enumerate}
\item
\label{T:fne}
{\bf (Lions and Mercier).}
$T_{(A,B)}$ is firmly nonexpansive.
\item
\label{T:selfdual}
{\bf (Eckstein).}
$T_{(A,B)} 
= T_{(A^{-1},B^{-\ovee})}.
$
\item
\label{fix:inc:Z}
{\bf(Combettes).} $Z=J_A(\fix T)$.
\item
\label{fix:inc:K}
$K=J_{A^{-1}}(\fix T).$
\end{enumerate}
\end{fact}
\begin{proof}
\ref{T:fne}: 
See, e.g., \cite[Lemma~1]{L-M79},
\cite[Corollary~4.2.1 on page~139]{EckThesis}, 
\cite[Corollary~4.1]{EckBer}, or
\cref{thm:simp:pr}\ref{thm:simp:pr:-i}
\ref{T:selfdual}: See e.g., 
\cite[Lemma~3.6 on page~133]{EckThesis} 
or \cite[Proposition~2.16]{JAT2012}).
\ref{fix:inc:Z}:
See \cite[Lemma~2.6(iii)]{Comb04}.
\ref{fix:inc:K}: 
See \cite[Corollary~4.9]{JAT2012}.
\end{proof}

The following notion, coined by Iusem \cite{Iusem98},
is very useful.
We say that $C:X\rras X$ 
is \emph{paramonotone}
if 
it is monotone and 
we have the implication
\begin{equation}
\left.
\begin{array}{c}
 (x,u)\in \gra C\\
(y,v)\in \gra C\\
\innp{x-y,u-v}=0
\end{array}
\right\}
\quad\RA\quad 
\big\{(x,v),(y,u)\big\}\subseteq \gra C.
\end{equation}

\label{fact:para:i}
\begin{example}
\label{ex:para:goodsub}
Let $f:X\to\left]-\infty,+\infty\right]$ be 
proper, convex and lower semicontinuous.
Then $\pt f$ is paramonotone by 
\cite[Proposition~2.2]{Iusem98} 
(or by \cite[Example~22.3(i)]{BC2011}).
\end{example}

We now recall that the so-called
``extended solution set''
(see \cite[Section~2.1]{EckSvai08}
and also \cite[Section~3]{JAT2012}) is defined by 
\begin{equation}
\label{def:ex:sol}
\SE := \SE_{(A,B)} :=\stb{(z,k)\in 
X\times X~|~-k\in Bz, k\in Az}\subseteq Z\times K.
\end{equation}
\begin{fact}
\label{fact:para:cc}
Recalling \cref{f:Mintypar}, 
we have the following:
\begin{enumerate}
\item
\label{fact:para:cc:ii}
$\SE = M(\Fix T) = 
\menge{(J_A\times J_{A^{-1}}) (y,y)}{y\in \fix T}$.
\item
\label{fact:para:cc:iii}
$\fix T=M^{-1}(\SE) = \menge{z+k}{(z,k)\in \SE}$.
\item
\label{fact:para:cc:i}
{(\bf Eckstein and Svaiter).} $\SE $ is closed and convex. 
\end{enumerate}
If $A$ and $B$ are paramonotone, 
then we additionally have:
\begin{enumerate}
 \setcounter{enumi}{3}
 \item
\label{fact:para:ii:b}
$\SE=Z\times K$.
\item
\label{fact:para:ii:a}
$\fix T=Z+K$.

\end{enumerate}
\end{fact}
\begin{proof}
\ref{fact:para:cc:ii}\&\ref{fact:para:cc:iii}:
This is \cite[Theorem~4.5]{JAT2012}.
\ref{fact:para:cc:i}:
See \cite[Lemma~2]{EckSvai08}.
Alternatively, 
since $\Fix T$ is closed, and $M$ and $M^{-1}$ are continuous, we
deduce the closedness from \ref{fact:para:cc:ii}.
The convexity was proved in \cite[Corollary~3.7]{JAT2012}.
\ref{fact:para:ii:b}~\&~\ref{fact:para:ii:a}:
See \cite[Corollary~5.5(ii)\&(iii)]{JAT2012}. 
\end{proof}

Working in $X\times X$,
we 
recall that (see, e.g., \cite[Proposition~23.16]{BC2011}) 
 \begin{equation}
 \label{e:prod:resolvent}
 \text{$A\times B$ is maximally monotone \; and }\;
 J_{A\times B}=J_A\times J_B.
 \end{equation}
 
 \begin{cor}
 \label{thm:simp:pr:d:iii}
Let $x\in X$ and let  $(z,k)\in \SE$. 
Then  
\begin{subequations}
\begin{align}
\normsq{(J_{A}T x,J_{A^{-1}} Tx)-(z,k)}
&=
\norm{J_{A}T x-z}^2+\norm{J_{A^{-1}} Tx-k}^2\\
&\le
\norm{J_{A} x-z}^2+\norm{J_{A^{-1}} x-k}^2\\
&=\normsq{(J_{A}x,J_{A^{-1}} x)-(z,k)}.
\end{align}
\end{subequations}
 \end{cor}
 \begin{proof}
 It follows from \cite[Theorem~4.5]{JAT2012}
 that $z+k\in \fix T$, $J_A(z+k)=z$ and 
 $J_{A^{-1}}(z+k)=k$.
Now combine with \cref{thm:simp:pr}\ref{thm:simp:pr:vv} with 
 $y$ replaced by $z+k$.
 \end{proof}

We recall, as consequence of 
\cite[Corollary~22.19]{BC2011}
and \cref{ex:para:goodsub},
 that when $X=\RR$,
the operators $A$ and $B$ are paramonotone.
In view of \cref{fact:para:cc}\ref{fact:para:ii:b}, 
we then have $\SE=Z\times K$. 

\begin{lem}
\label{lem:real:RR:static}
Suppose that $X=\RR$.
Let $x\in X$ and let $ (z,k)\in Z\times K$.
Then the following hold:
\begin{enumerate}
\item
\label{lem:real:1}
$\norm{J_{A} Tx-z}^2\le \norm{J_{A}x-z}^2$.
\item
\label{lem:real:2}
$\norm{J_{A^{-1}} Tx-k}^2\le\norm{J_{A^{-1}} x-k}^2$.
\end{enumerate}
\end{lem}

\begin{proof}
\ref{lem:real:1}:
Set
 \begin{equation}
 \label{eq:R:q}
q(x,z):=\norm{J_{A} Tx-z}^2- \norm{J_{A}x-z}^2.
\end{equation}
If $x\in \fix T$ we get $q(x,z)=0$. Suppose 
that $x\in \RR\smallsetminus \fix T$.
Since $T$ is firmly nonexpansive 
we have that $\Id-T$ is firmly nonexpansive 
(see \cite[Proposition~4.2]{BC2011}), 
hence monotone by 
\cite[Example~20.27]{BC2011}. 
Therefore $(\forall x\in \RR\smallsetminus \fix T)(\forall f\in \fix T)$ 
we have 
\begin{equation}\label{ine:fix}
{(x-Tx)(x-f)}=\bk{(\Id-T)x-(\Id-T)f}\bk{x-f}> 0.
\end{equation}
Notice that \eqref{eq:R:q} can be rewritten as 
 \begin{equation}\label{eq:R:q:i}
q(x,z)=(J_{A}T x-J_{A} x)(\bk{J_{A} Tx-z}+\bk{J_{A}x-z}).
\end{equation}
We argue by cases. 

\emph{Case~1:} $x< Tx$.\\
It follows from \eqref{ine:fix} that 
\begin{equation}\label{e:incr}
(\forall f\in \fix T)~x<f. 
\end{equation}
On the one hand, since $J_A$ is firmly nonexpansive, 
we have $J_A$ is monotone and 
therefore $J_A T x-J_A x\ge 0$.  
On the other hand, it follows from \cref{f:sd:ZK}\ref{fix:inc:Z} that
$(\exists f\in \fix T)$ such that $z=J_A f=J_A T f$.
Using \eqref{e:incr} and the fact that $J_A$ is monotone
we conclude that 
$J_{A} x-z=J_{A} x-J_{A} f\le 0$. 
Moreover, since $J_A$ and $T$ are firmly nonexpansive
operators on $\RR$, we have $J_A\circ T$ is 
firmly nonexpansive 
hence monotone and therefore
\eqref{e:incr} implies that 
$J_{A} Tx-z=J_{A} Tx-J_{A} Tf\le 0$. Combining with 
\eqref{eq:R:q:i}
we conclude that \ref{lem:real:1} holds.

\emph{Case 2:} $x>Tx$.\\
The proof follows similar to \emph{Case~1}.

\ref{lem:real:2}
Apply the results of \ref{lem:real:1} to $A^{-1}$
 and use \cref{e:iri}.
\end{proof}

In view of \cref{def:ex:sol}
one might conjecture that 
%\cref{thm:simp:pr}\ref{thm:simp:pr:iii}
\cref{thm:simp:pr:d:iii}
holds when we replace $\SE$
by $Z\times K$.
The following example gives a negative answer to this conjecture.
It also illustrates that when $X\neq \RR$, the conclusion of 
\cref{lem:real:RR:static} could fail.

 \begin{ex}\label{ex:cn:skew}
Suppose that $X=\RR^2$, that $A$ is 
the normal cone operator 
of $\RR^2_+$, and that 
$B:X\to X:(x_1,x_2)\mapsto (-x_2,x_1)$
is the rotator by $\pi/2$. Then $\fix T=\RR_+\cdot(1,-1)$, 
$Z=\RR_+\times\stb{0}$
and $K=\stb{0}\times\RR_{+}$. Moreover, 
 $(\exists x\in \RR^2)$ $(\exists (z,k)\in Z\times K)$ 
% such that
% $z+k\in (Z+K)\setminus \fix T$ 
such that 
 $\normsq{(J_{A}T x,J_{A^{-1}} Tx)-(z,k)}
 -\normsq{(J_{A} x,J_{A^{-1}} x)-(z,k)}>0$
  and $\normsq{J_ATx-z}-\normsq{J_Ax-z}= \tfrac{5}{4}a^2>0$.
\end{ex}
\begin{proof}
Let $(x_1,x_2)\in \RR^2$.
Using \cite[Proposition~2.10]{Sicon2014}
 we have 
  \begin{equation}
 J_B(x_1,x_2)=(\tfrac{1}{2}(x_1+x_2),\tfrac{1}{2}(-x_1+x_2))\;\;\;
 \text{and}\;\;\;R_B(x_1,x_2)=(x_2,-x_1)=-B(x_1,x_2).
 \end{equation}
 Hence, $R_B^{-1}=(-B)^{-1}=B$.
Using \eqref{def:T} we conclude that
$(x_1,x_2)\in\fix T\siff (x_1,x_2)\in\fix R_BR_A$. Hence
\begin{subequations}
\begin{align}
\fix T&=\menge{(x_1,x_2)\in \RR^2}{(x_1,x_2)=R_B R_A(x_1,x_2)}\\
&=\menge{(x_1,x_2)\in
\RR^2}{R_B^{-1}(x_1,x_2)=2J_A(x_1,x_2)-(x_1,x_2)}\\
&=\menge{(x_1,x_2)\in \RR^2}{B(x_1,x_2)+(x_1,x_2)=
2P_{\RR^2_+}(x_1,x_2)}\\
&=\menge{(x_1,x_2)\in \RR^2}{(x_1-x_2,x_1+x_2)=2P_{\RR^2_+}(x_1,x_2)}.
\end{align}
\end{subequations}

We argue by cases. 

\emph{Case~1:} $x_1\ge 0$ and  $x_2\ge 0$.
Then $(x_1,x_2)\in \fix T$
$\siff(x_1-x_2,x_1+x_2)=2P_{\RR^2_+}(x_1,x_2)=(2x_1,2x_2)$
$\siff x_1=-x_2$ and $x_1=x_2$ $\siff x_1=x_2=0$. 

\emph{Case~2:} $x_1< 0$ and  $x_2< 0$.
Then $(x_1,x_2)\in \fix T$
$\siff (x_1-x_2,x_1+x_2)=2P_{\RR^2_+}(x_1,x_2)=(0,0)$
$\siff x_1=x_2$ and $x_1=-x_2$ $\siff x_1=x_2=0$,
which contradicts that $x_1< 0$ and  $x_2< 0$.

\emph{Case~3:} $x_1\ge 0$ and  $x_2< 0$.
Then $(x_1,x_2)\in \fix T$
$\siff (x_1-x_2,x_1+x_2)=2P_{\RR^2_+}(x_1,x_2)=(2x_1,0)$
 $\siff x_1=-x_2$. %that is $\RR_+\cdot(1,-1)\subseteq \fix T$.

\emph{Case~4:} $x_1< 0$ and  $x_2\ge0$.
Then $(x_1,x_2)\in \fix T$
$\siff(x_1-x_2,x_1+x_2)=2P_{\RR^2_+}(x_1,x_2)=(0,2x_2)$
 $\siff x_1=x_2$, which never occurs
  since $x_1< 0$ and  $x_2\ge0$.
 Altogether we conclude that
$\fix T=R_+\cdot(1,-1)$, as claimed.

Using \cref{f:sd:ZK}\ref{fix:inc:Z}\&\ref{fix:inc:K} 
we have 
$
Z=J_A(\fix T)=\RR_+\times\stb{0},
$
and 
$
K=J_{A^{-1}}(\fix T)=(\Id-J_A)(\fix T)=\stb{0}\times\RR_-.
$

Now let $a>0$, let $x=(a,0)$, set $z:=(2a,0)\in Z$ 
and set $k:=(0,-a)\in K$.
Notice that 
$Tx=x-P_{\RR^2_+}+\frac{1}{2}(\Id-B)x
=(a,0)-(a,0)+\frac{1}{2}((a,0)-(0,a))=(\frac{1}{2}a,-\frac{1}{2}a)$.
Hence, $J_Ax=P_{\RR^2_+}(a,0)=(a,0)$, $\JAinv x=P_{\RR^2_-}(a,0)=(0,0)$,
$J_ATx=P_{\RR^2_+}(\frac{1}{2}a,-\frac{1}{2}a)=(\frac{1}{2}a,0)$, 
and $\JAinv x=P_{\RR^2_-}(\frac{1}{2}a,-\frac{1}{2}a)=(0,-\frac{1}{2}a)$.
Therefore
\begin{align*}
&\quad\normsq{(J_{A}T x,J_{A^{-1}} Tx)-(z,k)}
 -\normsq{(J_{A} x,J_{A^{-1}} x)-(z,k)}\\
&=\norm{J_{A}T x-z}^2+\norm{J_{A^{-1}} Tx-k}^2
-\norm{J_{A} x-z}^2-\norm{J_{A^{-1}} x-k}^2
\\
&=\normsq{(\tfrac{1}{2}a,0)-(2a,0)}+\normsq{(0,-\tfrac{1}{2}a)-(0,-a)}
-\normsq{(a,0)-(2a,0)}-\normsq{(0,0)-(0,-a)}\\
&=\tfrac{9}{4}a^2+\tfrac{1}{4}a^2-a^2-a^2=\tfrac{1}{2}a^2>0.
\end{align*}
Similarly one can verify that
$\normsq{J_ATx-z}-\normsq{J_Ax-z}=
  \tfrac{5}{4}a^2>0$.
\end{proof}

\section{Linear relations}

In this section, we assume that\footnote{$A\colon X\rras X$ 
is a \emph{linear relation} if 
$\gra A$ is a linear subspace of $X\times X$.}
\begin{empheq}[box=\mybluebox]{equation*}
A:X\rras X \text{~and~} B:X\rras X\text{~ are maximally monotone
linear relations};
\end{empheq} 
equivalently, by \cite[Theorem~2.1(xviii)]{BMW2}, that 
\begin{equation}
J_A ~\text{and }~J_B ~\text{~are linear operators from $X$ to $X$}. 
\end{equation}
This additional assumption leads to stronger conclusions.

\begin{lem}
\label{lem:lin:B}
$\Id-T ={J_A -2J_BJ_A+J_B}$.
\end{lem}
\begin{proof}
Let $x\in X$.
Then indeed $x-Tx=J_Ax-J_BR_Ax
=J_Ax-J_B(2J_Ax-x)
=J_Ax-2J_BJ_Ax+J_Bx$.
\end{proof}

\begin{lem}
Suppose that $U$ is a linear subspace of $X$
 and that $A=P_U$. 
 Then $A$ is maximally monotone, 
\begin{equation}
\label{eq:JPU}
J_A=J_{P_U}
=\tfrac{1}{2}(\Id+P_{U^\perp}),
\quad\text{and}\quad R_A=P_{U^\perp}=\Id-A.
\end{equation}
\end{lem}
\begin{proof}
Let $(x,y)\in X\times X$. Then 
\begin{equation}
\label{eq:PUX:PUY}
y=J_Ax\siff x=y+P_Uy.
\end{equation}
Now assume $y=J_Ax$.
Since $P_U$ is linear,
 \cref{eq:PUX:PUY} implies that $P_{U^\perp}x=P_{U^\perp}y$. 
Moreover, $y=x-P_Uy=\tfrac{1}{2}(x+x-2P_Uy)
 =\tfrac{1}{2}(x+y+P_Uy-2P_Uy)=\tfrac{1}{2}(x+y-P_Uy)
 =\tfrac{1}{2}(x+P_{U^\perp}y)=\tfrac{1}{2}(x+P_{U^\perp}x)$.
Next, $R_A=2J_A-\Id=(\Id+P_{U^\perp})-\Id=P_{U^\perp}$.
\end{proof}

We say that a linear relation $A$ is \emph{skew}
(see, e.g., \cite{BWY2010}) if $(\forall (a,a^*)\in \gra A)$ 
$\innp{a,a^*}=0$.
\begin{lem}
Suppose that $A:X\to X$
and $B:X\to X$ 
are both skew, and $A^2=B^2=-\Id$.
Then $\Id -T=\tfrac{1}{2}(\Id-BA)$.
\end{lem}
\begin{proof}
It follows from \cite[Proposition~2.10]{Sicon2014} 
that $R_A=A$ and $R_B=B$. Therefore
\cref{def:T} implies that $\Id-T=\tfrac{1}{2}(\Id-R_BR_A)=
\tfrac{1}{2}(\Id-BA)$.
\end{proof}

 \begin{example}
 \label{thm:simp:UV}
Suppose that $A$ and $B$ are skew. 
 Let $x\in X$ and let $y\in X$.
 Then the following hold:
 \begin{enumerate}
\item
\label{thm:simp:UV:i}
$
\innp{Tx-Ty,x-y}=\normsq{Tx-Ty}.
$

\item
\label{thm:simp:UV:ii}
$
\begin{aligned}[t]
\innp{(\Id-T)x-(\Id-T)y,x-y}&=\normsq{(\Id-T)x-(\Id-T)y}.
\end{aligned}
$

\item
\label{thm:simp:UV:iii}
$
\normsq{x-y}=\normsq{Tx-Ty}+\normsq{(\Id-T)x-(\Id-T)y}.
$
\item
\label{thm:simp:UV:iv}
$
\begin{aligned}[t]
&\norm{J_{A} x-J_{A}y}^2+\norm{J_{A^{-1}} x-J_{A^{-1}}y}^2
- \norm{J_{A}T x-J_{A}Ty}^2-\norm{J_{A^{-1}} Tx-J_{A^{-1}} Ty}^2\\
&=%\normsq{Tx-Ty}+
\normsq{(\Id-T)x-(\Id-T)y}.
\end{aligned}
$ 
%\item
%\label{thm:simp:UV:v}
%%\begin{equation}\label{e:ine:simp}
% $
%\norm{P_{U}T x-P_UTy}^2+\norm{P_{U^\perp} Tx-P_{U^\perp} Ty}^2
% \le \norm{P_{U} x-P_Uy}^2+\norm{P_{U^\perp} x-P_{U^\perp}y}^2.
%$
\item
\label{thm:simp:UV:vi}
$\normsq{x}=\normsq{Tx}+\normsq{x-Tx}$.
\item
\label{thm:simp:UV:vii}
$\innp{Tx, x-Tx}=0$.
\end{enumerate}
 \end{example} 
 \begin{proof}
\ref{thm:simp:UV:i}--\ref{thm:simp:UV:iv}: Apply
\cref{thm:simp:pr},
 and use \cref{Min:par} as well as the skewness of $A$ and $B$.
\ref{thm:simp:UV:vi}: Apply \ref{thm:simp:UV:iii} with $y=0$.
\ref{thm:simp:UV:vii}: We have
$2\innp{Tx, x-Tx}=\normsq{x}-\normsq{Tx}-\normsq{x-Tx}.$
Now apply \ref{thm:simp:UV:vi}. 
 \end{proof}

Suppose that $U$ is a closed affine subspace of $X$.
One can easily verify that 
\begin{equation} 
\label{eq:aff:orthog}
(\forall x\in X)(\forall y\in X)\quad 
\innp{P_U x-P_U y, (\Id-P_U)x-(\Id-P_U)y}=0.
\end{equation}

\begin{example}
Suppose that $U$ and $V$ are closed affine
subspaces of $X$ such that
$U\cap V\neq \fady$, that $A=N_U$, and that
$B=N_V$. 
Let $x\in X$, and 
let $(z,k)\in Z\times K)$.
Then 
\begin{subequations}
\begin{align}
&\hspace{-2cm}\normsq{(P_U x,(\Id-P_{U})x)-(z,k)}-
\normsq{(P_U Tx,(\Id-P_{U}) Tx)-(z,k)}\\
&=
\norm{x-(z+k)}^2
-\norm{Tx-(z+k)}^2
\label{eq:lin:p:i}\\ 
&=
\norm{x-Tx}^2
\label{eq:lin:p:i:i}\\
&=\norm{P_Ux-P_Vx}^2.
\label{eq:lin:p:ii}
%=&-\norm{P_Ux-P_Vx}^2.
\end{align}
\end{subequations}
 \end{example}
 \begin{proof}
As subdifferential operators,  $A$ and $B$ are paramonotone 
(by \cref{ex:para:goodsub}). 
It follows from \cref{fact:para:cc}\ref{fact:para:ii:a}
 and \cite[Theorem~4.5]{JAT2012} 
that
\begin{equation}
\label{eq:zk:infix}
z+k\in \fix T, \quad
P_U(z+k)=z\quad\text{and}\quad (\Id-P_{U})(z+k)=k.
\end{equation}
%therefore by \cite[Theorem~4.5]{JAT2012} 
%$P_U(z+k)=z$ and $P_{U^\perp}(z+k)=k$.
Hence, in view of \cref{eq:aff:orthog} we have
\begin{subequations}
\label{sub:e:i} 
\begin{align}
&\hspace{-2 cm}\normsq{(P_U x,(\Id-P_{U}) x)-(z,k)}\\
&=\norm{P_Ux-z}^2+\norm{(\Id-P_{U})x-k}^2 \\
&=\norm{P_Ux-P_U(z+k)}^2+\norm{(\Id-P_{U})x-(\Id-P_{U})(z+k)}^2\\
&\quad+2\innp{P_Ux-P_U(z+k),(\Id-P_{U}){x}-(\Id-P_{U}){(z+k)}}\\
&=\norm{P_Ux-P_U(z+k)+(\Id-P_{U}){x}
-(\Id-P_{U}){(z+k)}}^2\\
&=\norm{x-(z+k)}^2.
\end{align}
\end{subequations}
Applying \eqref{sub:e:i} with $x$ replaced by $Tx$ yields
\begin{equation}
\label{sub:e:ii}
\normsq{(P_U Tx,(\Id-P_{U})Tx)-(z,k)}=\norm{Tx-(z+k)}^2.%=\norm{P_UTx-z}^2
%+\norm{P_{U^{\perp}}Tx-k}^2.
\end{equation}

Combining \eqref{sub:e:i} and 
\eqref{sub:e:ii} yields \cref{eq:lin:p:i}.
%To prove \cref{eq:lin:p:i:i}, it is sufficient to
%show that $
%\norm{x-(z+k)}^2-\norm{Tx-(z+k)}^2=\norm{x-Tx}^2$.
It follows from \cref{eq:aff:orthog} and
 \cref{thm:simp:pr}\ref{thm:simp:pr:i}
 applied with $(A,B, y)$ replaced by $(N_U,N_V,z+k)$
that $\normsq{x-(z+k)}-\normsq{Tx-T(z+k)}=\normsq{x-Tx-((z+k)-T(z+k))}$,
which in view of \cref{eq:zk:infix}, proves \cref{eq:lin:p:i:i}.

Now we turn to \cref{eq:lin:p:ii}.
Let $w\in U\cap V$.
Then $U=w+\parl U$ and $V=w+\parl V$.
Suppose momentarily that $w=0$.
In this case, $\parl U=U$ and $\parl V=V$. 
%\ref{thm:simp:UV:viii}:
Using \cite[Proposition~3.4(i)]{JAT2014},
we have
\begin{equation}
\label{T:UV:back}
T=T_{(U,V)}=P_VP_U+P_{V^\perp}P_{U^\perp}.
\end{equation}
Therefore 
\begin{subequations}
\begin{align}
x-Tx
&=P_Ux+P_{U^\perp}x-P_VP_Ux-
P_{V^\perp}P_{U^\perp}x
=(\Id -P_V)P_Ux+(\Id-P_{V^\perp})
P_{U^\perp}x\\
&=P_{V^\perp}P_Ux+P_VP_{U^\perp}x.
\label{x:Tx:UV}
\end{align}
\end{subequations}
Using \eqref{x:Tx:UV} we have 
\begin{subequations}
\label{eq:lin:case}
\begin{align}
\norm{x-Tx}^2&=\norm{P_{V^\perp}
P_Ux+P_VP_{U^\perp}x}^2
=\norm{P_Ux-P_{V}P_Ux
+P_Vx-P_VP_{U}x}^2\\
&=\norm{P_Ux-2P_{V}P_Ux
+P_Vx}^2\\
&=\norm{P_Ux}^2+\norm{P_Vx}^2
+4\norm{P_VP_Ux}^2\\
&\qquad +2\innp{P_Ux,P_Vx}
-4\innp{P_Ux,P_VP_Ux}-4\innp{P_Vx,P_VP_Ux}\\
&=\norm{P_Ux}^2+\norm{P_Vx}^2
-2\innp{P_Ux,P_Vx}=\norm{P_Ux-P_Vx}^2.
\end{align}
\end{subequations}
Now, if $w\neq 0$, by \cite[Proposition~5.3]{BLM16}
we have  $Tx=T_{(\parl U,\parl V)}(x-w)+w$.
Therefore, \cref{eq:lin:case} yields
 $\normsq{x-Tx}=\normsq{(x-w)-T_{(\parl U,\parl V)}(x-w)}
 =\normsq{P_{\parl U}(x-w)-P_{\parl V}(x-w)}
 =\normsq{w+P_{\parl U}(x-w)-(w+P_{\parl V}(x-w))}=
 \norm{P_Ux-P_Vx}^2$,
 where the last equality follows from
  \cite[Proposition~3.17]{BC2011}.
 \end{proof}

\section{Main results}

In this section we consider the case when the set $Z$ is
possibly empty.

We recall the following important fact.
\begin{fact}[\bf{Infimal displacement vector}]
\label{F:v:WD}
{\rm (See, e.g., \cite{Ba-Br-Reich78},\cite{Br-Reich77} 
and \cite{Pazy}.)}
Let $T:X\to X$ be nonexpansive.
Then $\overline{\ran}(\Id-T)$ is convex;
consequently, the infimal displacement vector
\begin{empheq}[box=\mybluebox]{equation}
\label{eq:def:v}
\gap:=P_{\overline{\ran}(\Id-T)}.
\end{empheq}
is the unique and well-defined 
element in $\overline{\ran}{(\Id-T)}$
such that
$
\norm{\gap}=\ds\inf_{x\in X}\norm{x-Tx}.
$
\end{fact}
Following \cite{Sicon2014}, 
the \emph{normal problem} associated 
with the ordered pair 
$(A,B)$ is to\footnote{Let $w\in X$ be fixed. 
For the operator 
$A$, the \emph{inner and outer 
shifts} associated with $A$ are defined by 
$
\inns[w]{A}\colon X\To X \colon x\mapsto A(x-w)$
and
$
\outs[w]{A}\colon  X\To X \colon x\mapsto Ax-w.$
Note that $\inns[w]{A}$ and 
$\outs[w]{A}$ are maximally monotone.
}
\begin{equation}
\text{find $x\in X$
such that~}
0\in {_{\gap}A} x+B_{\gap}x=Ax-\gap+B(x-\gap).
\end{equation}
We shall use
\begin{equation}
Z_\gap:=Z_{({_{\gap}A} ,B_{\gap})}%=(A+B)^{-1}(0)
\qquad\text{and }\qquad 
K_\gap:=K_{(({_{\gap}A})^{-1} ,(B_{\gap})^{-\ovee})},%=(A^{-1}+B^{-\ovee})^{-1}(0),
\end{equation}
to denote the \emph{primal normal} and 
\emph{dual normal solutions}, respectively.
It follows from \cite[Proposition~3.3]{Sicon2014}
that
\begin{equation}
Z_v\neq \fady\siff v\in \ran(\Id-T).
\end{equation}

\begin{cor}
Let $x\in X$ and let $y\in X$.
Then the following hold:
\begin{subequations}
\label{eqs:series:summ}
\begin{align}
\sum_{n=0}^\infty\normsq{(\Id-T)T^nx-(\Id-T)T^ny}<+\infty,\\
\sum_{n=0}^\infty\underbrace{
\innp{J_AT^nx-J_AT^ny,\JAinv T^nx-\JAinv T^ny}}_{\ge 0}<+\infty,\\
\sum_{n=0}^\infty\underbrace{\innp{J_BR_AT^nx-J_BR_AT^ny,
J_{B^{-1}}R_A T^nx-J_{B^{-1}}R_AT^n y}}_{\ge 0}<+\infty.
\end{align}
\end{subequations}
Consequently, 
\begin{subequations}
\label{eqs:limits:zeros}
\begin{align}
{(\Id-T)T^nx-(\Id-T)T^ny}\to 0,\\
\innp{J_AT^nx-J_AT^ny,\JAinv T^nx-\JAinv T^ny}\to 0,
\label{e:golden:ineq}\\
\innp{J_BR_AT^nx-J_BR_AT^ny,
J_{B^{-1}}R_A T^nx-J_{B^{-1}}R_AT^n y}\to 0.
\end{align}
\end{subequations}
\end{cor}
\begin{proof}
Let $\nnn$. Applying 
\cref{Min:par}, to the points $T^n x$
 and $T^n y$, we learn that 
 $\stb{(J_AT^n x, J_{A^{-1}}T^n x), (J_AT^n y, J_{A^{-1}}T^n y)}
 \subseteq \gra A$, hence, by monotonicity of $A$ we have 
 $\innp{J_AT^nx-J_AT^ny,\JAinv T^nx-\JAinv T^ny}\ge 0$.
 Similarly 
 $\innp{J_BR_AT^nx-J_BR_AT^ny,J_{B^{-1}}R_A T^nx-J_{B^{-1}}R_AT^n y} \ge0$.
Now
\cref{eqs:series:summ} and 
\cref{eqs:limits:zeros}
 follow from \cref{thm:simp:pr}\ref{thm:simp:pr:i} by telescoping.
\end{proof}

The next result on \fejer\ monotone sequences 
is of critical importance in our analysis. 
(When $(u_n)_\nnn=(x_n)_\nnn$ one obtains a well-known result;
see, e.g., \cite[Theorem~5.5]{BC2011}.)

\begin{lemma}[\bf new \fejer\ monotonicity principle]
\label{Lem:sweet:lem}
Suppose that $E$ is a nonempty closed convex
subset of $X$, that $(x_n)_\nnn$ is a sequence in $X$
that is \emph{\fejer monotone with respect to $E$}, i.e.,
\begin{equation}
(\forall e\in E)(\forall\nnn)\quad\|x_{n+1}-e\|\leq\|x_n-e\|, 
\end{equation}
that $(u_n)_\nnn$ is a bounded sequence in $X$ such that
its weak cluster points lie in $E$, and that
\begin{equation}
\label{e:key:property}
(\forall e\in E)\;\;\innp{u_n-e,u_n-x_n}\to 0.
\end{equation}
Then $(u_n)_\nnn$ converges weakly 
to some point in $E$.
\end{lemma}
\begin{proof}
%Let $e_1\in E$ and $e_2\in E$.
It follows from \cref{e:key:property}
that
\begin{equation}
\label{e:key:zerolim}
(\forall (e_1,e_2)\in E\times E)\quad
\innp{e_2-e_1,u_n-x_n}=\innp{u_n-e_1,u_n-x_n}
-\innp{u_n-e_2,u_n-x_n}\to 0.
\end{equation}
Now obtain four subsequences $(x_{k_n})_\nnn$,
$(x_{l_n})_\nnn$, $(u_{k_n})_\nnn$ and
$(u_{l_n})_\nnn$ such that
$x_{k_n}\weak x_1$, $x_{l_n}\weak x_2$,
$u_{k_n}\weak e_1$ and $u_{l_n}\weak e_2$.
Taking the limit in \cref{e:key:zerolim} along these 
subsequences we have 
$\innp{e_2-e_1,e_1-x_1}=0=\innp{e_2-e_1,e_2-x_2}$,
hence
\begin{equation}
\label{e:final:conc}
\normsq{e_2-e_1}=\innp{e_2-e_1,x_2-x_1}.
\end{equation}
Since $\stb{e_1,e_2}\subseteq E$, 
we conclude, 
in view of \cite[Theorem~6.2.2(ii)]{Bau96} 
or \cite[Lemma~2.2]{BDM15:LNA}, that
 $\innp{e_2-e_1,x_2-x_1}=0$.
By \cref{e:key:zerolim}, $e_1=e_2$.
\end{proof}

We are now ready for our main result.

\begin{thm}[\bf shadow convergence]
\label{thm:abs:gen:shad}
% that
% $0\neq v\in \ran (\Id-T)$,
Suppose that $ x\in X$,
that the sequence $(J_A T^n x)_\nnn$
is bounded and its weak cluster points lie in $Z_v$,
that $Z_v\subseteq \fix(v+T)$
 and that $(\forall n\in \NN)$ $(\forall y\in \fix(v+T))$
 $J_A T^n y=y$.
 Then the ``shadow'' sequence $(J_A T^n x)_\nnn$
 converges weakly to some point in $Z_v$.
\end{thm}
\begin{proof}
Let $y\in \fix (v+T)$. Using \cref{e:golden:ineq}
and
\cite[Proposition~2.4(iv)]{BM15}
we have 
\begin{subequations}
\begin{align}
\innp{J_AT^n x - y, T^n x+nv-J_AT^n x}
&=\innp{J_AT^n x - y, T^n x-J_AT^n x-(y-nv-y)}\\
&=\innp{J_AT^n x -J_A T^n y, (\Id-J_A)T^n x-(\Id-J_A)T^n y}\\
&\to 0.
\end{align}
\end{subequations}
Note that \cite[Proposition~2.4(vi)]{BM15} implies that 
$(T^n x +nv)_\nnn$ is \fejer monotone 
with respect to $\fix (v+T)$ and
consequently with respect to $Z_v$.
Now apply  \cref{Lem:sweet:lem} with 
$E$ replaced by $Z_v$,
$(u_n)_\nnn$ replaced by
$(J_AT^n x)_\nnn$,
and $(x_n)_\nnn$ replaced by
$(T^n x +nv)_\nnn$.
\end{proof}

As a powerful application of \cref{thm:abs:gen:shad},
we obtain the following striking 
strengthening of a previous result on normal
cone operators.

\begin{thm}
\label{thm:nc:shad}
Suppose that $U$ and $V$ are nonempty closed 
convex subsets of $X$, that $A=N_U$,  
that $B=N_V$, that $\gap=P_{\overline{\ran}(\Id-T)}$ 
and that $U\cap(v+V)\neq \fady$. 
Let $x\in X$.
Then $(P_UT^n x)_\nnn$ converges weakly to 
some point in $Z_v=U\cap(v+V)$.
\end{thm}
\begin{proof}
%Note that $Z_v=U\cap(v+V)$.
It follows from \cite[Theorem~3.13(iii)(b)]{BCL04}
that $(P_UT^n x)_\nnn$ is bounded and 
its weak cluster points lie in $U\cap(v+V)$.
Moreover \cite[Theorem~3.5]{BCL04} implies that 
$Z_v=U\cap(v+V)\subseteq U\cap(v+V)
+N_{\overline{U-V}}(v)\subseteq \fix (v+T)$.
Finally,
 \cite[Lemma~3.12~\&~Proposition~2.4(ii)]{BCL04}
 imply that
$(\forall y\in \fix (v+T))(\forall \nnn)$ $P_UT^n y=P_U(y-nv)=y$,
 hence all the assumptions of \cref{thm:abs:gen:shad}
 are satisfied and the result follows.
\end{proof}
\begin{rem}
\label{rem:BCL04}
Suppose that  $v\in \ran(\Id -T)$.
More than a decade ago, it was shown in \cite{BCL04}
that when $A=N_U$ and $B=N_V$,
where $U$ and $V$ are nonempty closed convex subsets 
of $X$, that $(P_UT^n x)_\nnn$ is bounded and its weak cluster points 
lie in $U\cap(v+V)$. \cref{thm:nc:shad} 
yields the much stronger result
that $(P_UT^n x)_\nnn$
converges weakly to a point in $U\cap(v+V)$.
\end{rem}

Here is another instance of \cref{thm:abs:gen:shad}. 

\begin{ex}
\label{ex:gen:incon}
Suppose that $U$ is a closed 
linear subspace of
$X$, that $b\in U^\perp\smallsetminus \stb{0}$, that $A=N_U$
and that $B=\Id+N_{(-b+U)}$. 
Then $Z=\fady$, $v=b\in \ran(\Id-T)$,
$Z_v=\stb{0}$ and $K_v=U^\perp$.
Moreover,
$(\forall x\in X)$ $(\forall n\in \NN)$
$P_UT^n x=\tfrac{1}{2^n}P_Ux\to 0$ and
$\norm{P_{U^\perp} T^n x}\to \infty$.
\end{ex}
\begin{proof}
By the Brezis-Haraux theorem 
(see \cite[Theorems~3~\&~4]{Br-H} or \cite[Theorem~24.20]{BC2011})
we have $X=\intr X
\subseteq \intr \ran B 
=
\intr (\ran \Id+\ran N_{(-b+U)})
\subseteq X$,
hence $\ran B=X$.
Using \cite[Corollary~5.3(ii)]{BHM15} we have 
$\overline{\ran}(\Id-T)
=\overline{(\dom A-\dom B)}\cap\overline{(\ran A+\ran B)}
=(U+b-U)\cap(U^\perp+X)=b+U$. 
Consequently, using \cite[Definition~3.6]{Sicon2014} 
 and \cite[Proposition~3.17]{BC2011} we have
\begin{equation}\label{ex:e:loc:gap}
v=P_{\overline{\ran}(\Id-T)}0=P_{b+U}0=b+P_{U}(-b)=b\in U^\perp\smallsetminus \stb{0}.
\end{equation}
Note that $\dom \outs[v]{A}=\dom A=U$
and $\dom \inns[v]{B}=v+\dom B=b-b+U=U$,
hence $\dom ( \outs[v]{A}+\inns[v]{B})=U\cap U=U$.
Let $x\in U$. Using \cref{ex:e:loc:gap}  
% \cite[Proposition~3.17)]{BC2011} 
we have
\begin{subequations}
\begin{align}
x\in Z_v&\siff 0\in N_{U}x-b+x-b+N_{-b+U}(x-b)
= N_{U}x-b+x-b+N_{U}x\\
&\siff 0\in U^{\perp}-b+x-b+U^\perp=x+U^\perp\siff x\in U^\perp,
\end{align}  
\end{subequations}
  hence $Z_v=\stb{0}$, as claimed.
As subdifferentials, both $A$ and $B$
  are paramonotone, and 
  so are the translated operators
  $\outs[v]{A}$ and $\inns[v]{B}$.  
Since  $ Z_v=\stb{0}$, in view of \cite[Remark~5.4]{JAT2012}  
and \cref{ex:e:loc:gap} we learn that 
\begin{equation}
K_v=(N_U 0-b)\cap(0-b+N_{-b+U}(0-b))=(U^\perp-b)\cap(-b+U^\perp)=U^\perp.
\end{equation}
Next we claim that 
\begin{equation}
\label{eq:PUTX}
(\forall x\in X)\quad P_U Tx=\tfrac{1}{2}P_U x.
\end{equation}
Indeed, note that $J_B=(\Id+B)^{-1}
=(2\Id+N_{-b+U})^{-1}
=(2\Id+2N_{-b+U})^{-1}
=(\Id+N_{-b+U})^{-1}\circ(\tfrac{1}{2}\Id)
=P_{-b+U}\circ(\tfrac{1}{2}\Id)=-b+\tfrac{1}{2}P_{U}$,
where the last identity follows from \cite[Proposition~3.17)]{BC2011}
 and  \cref{ex:e:loc:gap}. 
Now, using that\footnote{It follows from \cite[Corollary~3.20]{BC2011}
that $P_U$ is linear, hence, $P_UR_U=P_U(2P_U-\Id)=2P_U-P_U=P_U$.} 
$P_UR_U=P_U$ and \cref{ex:e:loc:gap} we have 
$P_U T
=P_U(P_{U^\perp} +J_BR_U)
=P_UJ_BR_U
=P_U(-b+\tfrac{1}{2}P_{U})R_U
=P_U(-b+\tfrac{1}{2}P_{U})
=\tfrac{1}{2}P_U$.
To show that 
$(\forall x\in X)$ $(\forall n\in \NN)$
$P_UT^n x=\tfrac{1}{2^n}P_Ux$, we use induction.
Let $x\in X$. Clearly, when $n=0$, the base case holds.
Now suppose that for some $n\in \NN$, we have,
for every $x\in X$,
$P_UT^n x=\tfrac{1}{2^n}P_Ux$.
Now applying the inductive hypothesis with
$x$ replaced by $Tx$,
and using \cref{eq:PUTX},
 we have 
$P_UT^{n+1} x
=P_UT^n (T x)
=\tfrac{1}{2^n}P_U Tx
=\tfrac{1}{2^n}P_U(\tfrac{1}{2}P_Ux)
=\tfrac{1}{2^{n+1}}P_U x$,
 as claimed. Finally, using \cref{ex:e:loc:gap} and 
 \cite[Corollary~6(a)]{Pazy} we have $\norm{T^n x}\to +\infty$,
 hence
\begin{equation}
\normsq{P_{U^\perp}T^n x}=\normsq{T^n x}-\normsq{P_UT^n x}
=\normsq{T^n x}-\tfrac{1}{4^n}\|P_U x\|^2\to+\infty.
%=\normsq{T^n x}-\normsq{P_{U}(T^n x+nv)},
\end{equation}
\vspace{-5mm}
\end{proof}

In fact, as we shall now see, the shadow
sequence may be unbounded in the general case, even 
when one of the operators is a normal cone operator.
\begin{rem}{\bf{(shadows 
in the presence of normal solutions)}}
\label{rem:gen:case}
\begin{enumerate}
\item
\label{rem:i}
\cref{ex:gen:incon} illustrates that 
even when normal solutions exists, the shadows need 
not converge.
Indeed, we have $K_v=U^\perp\neq \fady$
but the \emph{dual} shadows satisfy
 $\norm{P_{U^{\perp}} T^n x}\to +\infty$.
 \item 
 Suppose that $A$ and $B$ are as defined in 
 \cref{ex:gen:incon}. Set $\widetilde{A}=A^{-1}$,
 $\widetilde{B}=B^{-\ovee}$ and $\widetilde{Z}=
 Z_{(\widetilde{A},\widetilde{B})}$. 
By \cite[Proposition~2.4(v)]{JAT2012} 
 $\widetilde{Z}\neq \fady\siff K_{(\widetilde{A},\widetilde{B})}=Z_{(A,B)}\neq \fady$,
 hence $\widetilde{Z}= \fady$. Moreover, 
 \cite[Remarks~3.13~\&~3.5]{Sicon2014} 
 imply that
 $v=b\in \ran(\Id-T)$
and  $\widetilde{Z}_v=U^\perp+b=U^\perp\neq \fady$.
However, in the light of \ref{rem:i} the \emph{primal} shadows
satisfy
$\norm{J_{\widetilde{A}}T^n x}=\norm{J_{A^{-1}}T^n x}
=\norm{P_{U^{\perp}} T^n x}\to +\infty$.
\item
Concerning \cref{thm:abs:gen:shad},
it would be interesting to find other conditions sufficient for
weak convergence of the shadow sequence or to even characterize
this behaviour. 
 \end{enumerate}
\end{rem}

\section{A proof of the Lions-Mercier-Svaiter theorem}

In this section, we work under the assumptions that
\begin{equation}
Z\neq \fady 
\quad\text{and} \quad 
\fix T\neq \fady.
\end{equation}

Parts of the following two results are implicit in
\cite{Svaiter}; however, our proofs are different. 

\begin{prop}
\label{cor:cluster:fen}
Let $x\in X$. Then the following hold:
\begin{enumerate}
\item
\label{cor:cluster:fen:i}
$T^n x-T^{n+1} x=J_AT^n x -J_BR_AT^n x
=J_{A^{-1}}T^n x +J_{B^{-1}}R_AT^n x
\to 0$.
\item
\label{cor:cluster:fen:ii}
The sequence $(J_AT^n x,J_BR_AT^n x,J_{A^{-1}}T^n x,
J_{B^{-1}}R_AT^n x)_\nnn$
 is bounded and lies in  $\gra (A\times  B)$.
%\item
%\label{cor:cluster:fen:iii}
\end{enumerate}
Suppose that $(a,b,a^*,b^*)$ is
 a weak cluster point of 
 $(J_AT^n x,J_BR_AT^n x,J_{A^{-1}}T^n x,
J_{B^{-1}}R_AT^n x)_\nnn$.
Then:
\begin{enumerate}
\setcounter{enumi}{2}
\item
\label{eq:a=b}
$
a-b=a^*+b^*=0.$ 
\item
\label{eq:a=-b}
$\innp{a,a^*}+\innp{b,b^*}=0.$
\item
\label{eq:in:gra}
$\bk{a,a^*}\in \gra A$
 and 
$\bk{b,b^*}\in \gra B.$
\item
\label{cor:cluster:fen:conc}
For every $x\in X$, the
sequence $(J_AT^n x,J_{A^{-1}}T^n x)_\nnn$
is bounded and its weak cluster points lie in $\SE$.
\end{enumerate}
\end{prop}
\begin{proof}
\ref{cor:cluster:fen:i}:
Apply
\cref{lem:cluster:fen}\ref{lem:cluster:fen:i} 
with $x$ replaced by $T^n x$. 
The claim of the strong limit follows from 
combining \cref{f:sd:ZK}\ref{T:fne} and \cite[Corollary~2.3]{Ba-Br-Reich78} or
\cite[Theorem~5.14(ii)]{BC2011}.

\ref{cor:cluster:fen:ii}:
The boundedness of the sequence 
follows from the weak convergence of 
$(T^n x)_\nnn$ (see, e.g.,\cite[Theorem~5.14(iii)]{BC2011})
and the nonexpansiveness 
of the resolvents and reflected resolvents 
of monotone operators
(see, e.g., \cite[Corollary~23.10(i) and (ii)]{BC2011}).
Now apply \cref{lem:cluster:fen}\ref{lem:cluster:fen:ii}
with $x$ replaced by $T^n x$.
\ref{eq:a=b}: This follows from taking the weak limit 
along the subsequences 
in \ref{cor:cluster:fen:i}. 
 \ref{eq:a=-b}: In view of \ref{eq:a=b} we have %$a=b$,hence 
 $\innp{a,a^*}+\innp{b,b^*}=\innp{a, a^*+b^*}=\innp{a,0}=0$.
\ref{eq:in:gra}:
Let $((x,y),(u,v))\in \gra (A\times B)$
and set
\begin{equation}
\label{e:def:seq}
a_n:=J_A T^n x, a^*_n:=J_{A^{-1}}T^nx, 
b_n:=J_BR_A T^n x, b^*_n:=J_{B^{-1}}R_AT^nx.
\end{equation}
Applying \cref{lem:abs:8} with
$(a,b,a^*,b^*)$
replaced by 
$(a_n,b_n,a^*_n,b^*_n)$ yields
\begin{align}
\label{eq:apply:lem}
 \innp{(a_n,b_n)-(x,y),(a^*_n,b^*_n)-(u,v)}&
=\innp{\at_n-\bt_n,\ut_n}%-\innp{\bt,a^*_n}
+\innp{\xt,u}
-\innp{\xt,\ut_n}
-\innp{\at_n-\bt_n,u}%+\innp{\bt,\xt^*}
\nonumber\\
&\quad+\innp{\bt_n,\ut_n+\vt_n}%+\innp{\bt,a^*_n}
+\innp{\yt,v}
-\innp{\yt,\vt_n}
-\innp{\bt_n,u+v}.%-\innp{b_n,x_a^*}
\end{align}
By \cref{e:prod:resolvent}, 
$A\times B$ is monotone. 
In view of 
\cref{e:def:seq},
\cref{eq:apply:lem}
%\cref{Min:par} 
and \cref{cor:cluster:fen}\ref{cor:cluster:fen:ii},
we deduce that  
\begin{align}
\label{eq:maxm:ve:d}
&\innp{\at_n-\bt_n,\ut_n}%-\innp{\bt,a^*_n}
+\innp{\xt,u}
-\innp{\xt,\ut_n}
-\innp{\at_n-\bt_n,u}%+\innp{\bt,\xt^*}
\nonumber\\
&+\innp{\bt_n,\ut_n+\vt_n}%+\innp{\bt,a^*_n}
+\innp{\yt,v}
-\innp{\yt,\vt_n}
-\innp{\bt_n,u+v}\ge 0.
\end{align}
Taking the limit in \cref {eq:maxm:ve:d} along a subsequence
  and using \cref{e:def:seq}, \cref{cor:cluster:fen}\ref{cor:cluster:fen:i},
  \ref{eq:a=b} and \ref{eq:a=-b} yield
  \begin{align}
  0&\le \innp{x,u}
-\innp{x,a^*}
+\innp{y,v}
-\innp{y,b^*}-\innp{b,u+v}\nonumber\\
&=\innp{x,u}
-\innp{x,a^*}
+\innp{y,v}
-\innp{y,b^*}-\innp{a,u}-\innp{b,v}
+\innp{a,a^*}+\innp{b,b^*}\nonumber\\
&=\innp{a-x,a^*-u}+\innp{b-y,b^*-v}
=\innp{(a,b)-(x,y),(a^*,b^*)-(u,v)}.
  \end{align}
By maximality of $A\times B$ (see \cref{e:prod:resolvent})
we deduce that
$((a,b),(a^*,b^*))\in \gra (A\times B)$. 
Therefore, $(a,a^*)\in \gra A$
 and  $(b,b^*)\in \gra B$.
 \ref{cor:cluster:fen:conc}:
 The boundedness of the sequence follows from
%\cref{cor:cluster:fen}
\ref{cor:cluster:fen:ii}.
Now 
let $(a,b,a^*,b^*)$
be a weak cluster point of  
$((J_AT^n x,J_BR_AT^n x,J_{A^{-1}}T^n x,
J_{B^{-1}}R_AT^n x))_\nnn$.
By %\cref{cor:cluster:fen}
\ref{eq:in:gra} %and \ref{eq:a=b} 
we know that $(a,a^*)\in \gra A$
 and $(b,b^*)=(a,b^*)\in \gra B$, which in view of 
\ref{eq:a=-b} implies
 $a^*\in Aa$ and $-a^*=b^*\in Bb=Ba$, hence $(a,a^*)\in \SE$,
 as claimed (see \cref{def:ex:sol}).
\end{proof}

\begin{thm}
\label{thm:simp:pr:d}
Let $x\in X$ and let $(z,k)\in \SE$.
Then  
%$(\forall x\in X)(\forall y\in X)$ 
the following hold:

\begin{enumerate}
\item
\label{thm:simp:pr:d:iii-1}
For every $\nnn$, %$(\forall n\in\NN)$
\begin{subequations}
\begin{align}
\normsq{(J_{A}T^{n+1} x,J_{A^{-1}} T^{n+1}x)-(z,k)}
&=
\norm{J_{A}T^{n+1} x-z}^2+\norm{J_{A^{-1}} T^{n+1}x-k}^2\\
&\le 
\norm{J_{A} T^{n}x-z}^2+\norm{J_{A^{-1}} T^{n}x-k}^2\\
&=\normsq{(J_{A}T^{n}x,J_{A^{-1}} T^{n}x)-(z,k)}.
\end{align}
\end{subequations}
\item
\label{thm:simp:pr:d:iii-2}
The sequence $(J_AT^n x, J_{A^{-1}}T^nx)_\nnn$
is Fej\'{e}r monotone with respect to $\SE$.
\item\label{thm:simp:pr:iv} 
The sequence $(J_AT^n x, J_{A^{-1}}T^nx)_\nnn$
converges weakly to some point in $\SE$.
\end{enumerate}
\end{thm}
\begin{proof}
  \ref{thm:simp:pr:d:iii-1}:
 Apply \cref{thm:simp:pr:d:iii} with $x$
 replaced by $T^n x$.
 \ref{thm:simp:pr:d:iii-2}:
 This follows directly from 
 \ref{thm:simp:pr:d:iii-1}.
 \ref{thm:simp:pr:iv}:
Combine  \cref{cor:cluster:fen}\ref{cor:cluster:fen:conc}, 
\ref{thm:simp:pr:d:iii-2}, 
\cref{fact:para:cc}\ref{fact:para:cc:i}
 and \cite[Theorem~5.5]{BC2011}.
\end{proof}

\begin{cor}{\bf(Lions--Mercier--Svaiter).}
$(J_AT^n x)_\nnn$
converges weakly to some point in $Z$.
\end{cor}
\begin{proof}
This follows from \cref{thm:simp:pr:d}\ref{thm:simp:pr:iv}; see also
Lions and Mercier's \cite[Theorem~1]{L-M79} and Svaiter's 
\cite[Theorem~1]{Svaiter}.
\end{proof}

\begin{rem}[\bf brief history]
\label{rem:history}
The Douglas--Rachford algorithm has its roots in the 1956 paper 
\cite{DoRa} as a method for solving a system of linear equations.
Lions and Mercier, 
in their brilliant seminal work \cite{L-M79} from 1979, presented a
broad and powerful generalization to its current form. 
(See \cite{BLM16} and \cite{Comb04} for details on this connection.)
They showed that $(T^nx)_\nnn$ converges weakly to a point in $\Fix
T$ and that the bounded shadow sequence $(J_AT^nx)_\nnn$ has all
its weak cluster points in $Z$ provided that $A+B$ was maximally
monotone. 
(Note that resolvents are \emph{not} weakly continuous in general;
see, e.g., 
\cite{Zar71:1} or \cite[Example~4.12]{BC2011}.)
Building on \cite{Bau09} and \cite{EckSvai08}, 
Svaiter provided a beautiful complete answer in 2011 (see \cite{Svaiter})
demonstrating that $A+B$ does not have to be maximally monotone
and that the shadow sequence $(J_AT^nx)_\nnn$ in fact does
converge weakly to a point in $Z$. 
(He used
\cref{thm:simp:pr:d}; however, his proof differs from ours which
is more in the style of the original paper by Lions and Mercier
\cite{L-M79}.)
Nonetheless, when $Z=\varnothing$, 
the complete understanding of $(J_AT^nx)_\nnn$ 
remains open --- to the best of our knowledge,
\cref{thm:abs:gen:shad} is currently the most
powerful result available. 
\end{rem}

In our final result we show that 
when $X=\RR$, the \fejer monotonicity of 
the sequence $(J_AT^n x, J_{A^{-1}}T^n x)_\nnn$
with respect to $S$ can be decoupled to yield 
\fejer monotonicity of $(J_AT^n x)_\nnn$
 and $( J_{A^{-1}}T^n x)_\nnn$ with respect to 
 $Z$ and $K$, respectively.

\begin{lem}
\label{lem:real}
Suppose that $X=\RR$.
Let $x\in X$ and let $ (z,k)\in Z\times K$.
Then the following hold:
\begin{enumerate}
\item
\label{lem:real:1:dyn}
The sequence $(J_A T^n x)_\nnn$
is \fejer monotone with respect to $Z$.
\item
\label{lem:real:2:dyn}
The sequence $(J_{A^{-1}} T^n x)_\nnn$
is \fejer monotone with respect to $K$.
\end{enumerate}
\end{lem}
\begin{proof}
Apply \cref{lem:real:RR:static} with $x$
replaced by $T^n x$.
\end{proof}

We point out that the conclusion of \cref{lem:real}
does not hold when $\dim X\ge 2$, 
see \cite[Section~5~\&~Figure~1]{JAT2014}.
\section*{Acknowledgments}
HHB was partially supported by the Natural Sciences and Engineering Research Council of Canada and by the Canada Research Chair Program.


\small 

\begin{thebibliography}{999}
\seppfive
\bibitem{AT}
H.\ Attouch and M.\ Th\'era,
A general duality principle for the sum of two operators,
\emph{Journal of Convex Analysis} 3 (1996), 1--24.

\bibitem{Ba-Br-Reich78}
J.B.\ Baillon, R.E.\ Bruck and S.\ Reich, 
On the asymptotic behavior of nonexpansive 
mappings and semigroups in Banach spaces, 
\emph{Houston Journal of Mathematics} 4 (1978), 1--9.

\bibitem{Bau96}
H.H.\ Bauschke, 
\emph{Projection Algorithms and Monotone Operators}, 
PhD thesis, Simon Fraser University, Burnaby, B.C., Canada, August 1996.

\bibitem{Bau09}
H.H.\ Bauschke,
A note on the paper by Eckstein and Svaiter on 
``General projective splitting for sums of maximal monotone
operators'',
\emph{SIAM Journal on Control and Optimization}~48 (2009),
2513--2515. 

\bibitem{JAT2014}
H.H.\ Bauschke,  J.Y.\ Bello Cruz, 
T.T.A.\ Nghia, H.M.\ Phan, and X.\ Wang, 
The rate of linear convergence of the 
Douglas--Rachford algorithm for subspaces 
is the cosine of the Friedrich's angle, 
\emph{Journal of Approximation Theory} 185~(2014), 63--79.


\bibitem{Sicon2014}
H.H.\ Bauschke, W.L.\ Hare, and W.M.\ Moursi, Generalized solutions 
for the sum of two maximally monotone operators, 
\emph{SIAM Journal on Control and Optimization}, 52~(2014), 1034--1047.


\bibitem{JAT2012} H.H.\ Bauschke, R.I.\ Bo\c{t}, 
W.L.\ Hare and W.M.\ Moursi,
Attouch--Th\'era 
duality revisited: paramonotonicity and operator splitting, 
\emph{Journal of Approximation Theory}, 164~(2012), 1065--1084. 

\bibitem{BC2011}
H.H.\ Bauschke and P.L.\ Combettes,
\emph{Convex Analysis and Monotone 
Operator Theory in Hilbert Spaces},
Springer, 2011.

\bibitem{BCL04}
H.H.\ Bauschke, P.L.\ Combettes, and D.R.\ Luke,
Finding best approximation pairs relative 
to two closed convex sets in Hilbert spaces,
\emph{Journal of Approximation Theory}~127 (2004), 
178--192.

\bibitem{BDM15:LNA}
H.H.\ Bauschke, M.N.\ Dao, and W.M.\ Moursi,
On \fejer monotone sequences and nonexpansive mappings,
\emph{Linear and Nonlinear Analysis}~1 (2015),
287--295.

\bibitem{BLM16}
H.H.\ Bauschke, Brett\ Lukens and W.M.\ Moursi, 
Affine nonexpansive operators, Attouch--Th\'{e}ra 
duality and the Douglas--Rachford algorithm,
\href{http://arxiv.org/pdf/1603.09418v1.pdf}{\texttt{arXiv:1603.09418 [math.OC]}}.
   
%
\bibitem{BMW2}
H.H.\ Bauschke, S.M.\ Moffat, and X.\ Wang,
Firmly nonexpansive mappings and maximally monotone operators: 
correspondence and duality, 
\emph{Set-Valued and Variational Analysis}~20 (2012), 131--153.

\bibitem{BM15}
H.H.\ Bauschke and W.M.\ Moursi, The Douglas--Rachford 
algorithm for two (not necessarily intersecting) 
affine subspaces, to appear in 
\emph{SIAM Journal on Optimization},
\href{http://arxiv.org/pdf/1504.03721v1.pdf}{\texttt{arXiv:1504.03721 [math.OC]}}.


\bibitem{BHM15}
H.H.\ Bauschke, W.L.\ Hare, and W.M.\ Moursi,
On the range of the Douglas--Rachford operator, to 
appear in \emph{Mathematics of Operation Research},
\href{http://arxiv.org/pdf/1405.4006.pdf}{\texttt{arXiv:1405.4006v2 [math.OC]}}.

\bibitem{BWY2010}
H.H.\ Bauschke, X.\ Wang, and L.\ Yao,
 Examples of discontinuous maximal 
 monotone linear operators and the 
 solution to a recent problem posed by B.F.\ Svaiter, 
 \emph{Journal of Mathematical 
 Analysis and Applications}~370 (2010), 224--241.
 
 \bibitem{Borwein50}
J.M.\ Borwein,
Fifty years of maximal monotonicity,
\emph{Optimization Letters}~4 (2010), 473--490.


\bibitem{Br-H} 
H.\ Brezis, A.\ Haraux, 
Image d'une Somme d'op\'{e}rateurs Monotones et Applications,
\emph{Israel Journal of Mathematics} 23 (1976), 165--186.


\bibitem{Brezis}
H.\ Brezis,
\emph{Operateurs Maximaux Monotones et
Semi-Groupes de Contractions dans les Espaces de Hilbert},
North-Holland/Elsevier, 1973. % New York
\bibitem{Br-Reich77}
R.E.\ Bruck and S.\ Reich, 
Nonexpansive projections and resolvents of accretive operators in
Banach spaces, \emph{Houston Journal of Mathematics} 
3 (1977), 459--470.


\bibitem{BurIus}
R.S.\ Burachik and A.N.\ Iusem,
\emph{Set-Valued Mappings and Enlargements
of Monotone Operators},
Springer-Verlag, 2008.

\bibitem{Comb04}
P.L.\ Combettes,
Solving monotone inclusions via compositions of nonexpansive averaged
operators,
\emph{Optimization}~53 (2004), 475--504.

\bibitem{Comb96}
P.L.\ Combettes,
The convex feasibility problem in image recovery,
\emph{Advances in Imaging and Electron Physics}~25 (1995), 155--270.


\bibitem{DoRa}
J.\ Douglas and H.H.\ Rachford,
On the numerical solution of the heat conduction
problem in 2 and 3 space variables,
\emph{Transactions of the AMS}~82 (1956), 421--439.

\bibitem{EckThesis}
J.\ Eckstein,
\emph{Splitting Methods for Monotone Operators with
Applications to Parallel Optimization},
Ph.D.~thesis, MIT, 1989.
%
\bibitem{EckBer}
J.\ Eckstein and D.P.\ Bertsekas,
On the Douglas--Rachford splitting method
and the proximal point algorithm for maximal monotone
operators,
\emph{Mathematical Programming}~55 (1992), 293--318.

\bibitem{EckSvai08}
J.\ Eckstein and B.F.\ Svaiter,
A family of projective splitting methods for the sum
of two maximal monotone operators,
\emph{Mathematical Programming (Series~B)}~111 (2008), 173--199.
%
%%
\bibitem{L-M79}
 P.L.\ Lions and B.\ Mercier, Splitting algorithms for the sum of two
nonlinear operators,
\emph{SIAM Journal on Numerical Analysis} 16~(1979), 964--979. 

\bibitem{Iusem98}
A.N.\ Iusem,
On some properties of paramonotone operators,
\emph{Journal of Convex Analysis}~5 (1998), 269--278.

\bibitem{Minty}
G.J.\ Minty,
Monotone (nonlinear) operators in Hilbert spaces,
\emph{Duke Mathematical Journal} 29 (1962), 341--346.


\bibitem{Pazy} A. Pazy, Asymptotic behavior of contractions in Hilbert space, 
\emph{Israel Journal of Mathematics} 9 (1971), 235--240.

%
\bibitem{Rock98}
R.T.\ Rockafellar and R. J-B\ Wets,
\emph{Variational Analysis},
Springer-Verlag, %New York,
corrected 3rd printing, 2009.
%%
\bibitem{Simons1}
S.\ Simons,
\emph{Minimax and Monotonicity},
Springer-Verlag,
1998.

\bibitem{Simons2}
S.\ Simons,
\emph{From Hahn-Banach to Monotonicity},
Springer-Verlag,
2008.
%
\bibitem{Svaiter}
B.F.\ Svaiter,
On weak convergence of the Douglas--Rachford method,
\emph{SIAM Journal on Control and Optimization}
49 (2011), 280--287.
%

\bibitem{Zar71:1}
E.H. Zarantonello, 
Projections on convex sets in 
Hilbert space and spectral theory, in: E.H.
Zarantonello (Ed.), \emph{Contributions to Nonlinear 
Functional Analysis}, Academic Press, New York,
(1971),  237--424.

\bibitem{Zeidler2a}
E.\ Zeidler,
\emph{Nonlinear Functional Analysis and Its Applications II/A:
Linear Monotone Operators},
Springer-Verlag, 1990.

\bibitem{Zeidler2b}
E.\ Zeidler,
\emph{Nonlinear Functional Analysis and Its Applications II/B:
Nonlinear Monotone Operators},
Springer-Verlag, 1990.

\end{thebibliography}
\end{document}